\documentclass[UKenglish, a4paper, 12pt, notitlepage]{article}
\usepackage[utf8]{inputenc}
\usepackage{amsmath, amssymb, latexsym, amsfonts, amscd, amsthm}
\usepackage[T1]{fontenc}
\usepackage{mathrsfs}
\usepackage{bbm}
\usepackage{babel}
\usepackage[all]{xy}
\usepackage{mathabx}
\usepackage{bbold}
\usepackage[scr=boondoxo]{mathalpha}
\usepackage{enumitem}
\usepackage[a4paper,bindingoffset=0.2in,%
            left=2cm,right=2cm,top=1in,bottom=2.5cm,%
            footskip=1.5cm]{geometry}

\newcommand{\drept}[1]{\ensuremath{\mathrm{#1}}}
\newcommand{\rond}[1]{\ensuremath{\mathscr{#1}}}
\newcommand{\eqdef}{\ensuremath{\overset{\text{def}}{=}}}
\newcommand{\hm}[3]{\ensuremath{\mathrm{Hom}_{#1}\left(#2, #3\right)}}
\newcommand{\iso}[3]{\ensuremath{\mathrm{Iso}_{#1}\left(#2, #3\right)}}
\newcommand{\id}{\ensuremath{\mathbbm{1}}}
\newcommand{\gr}{\ensuremath{\mathbf{Gr}}}
\newcommand{\ens}{\ensuremath{\mathbf{Ens}}}
\newcommand{\Om}{\ensuremath{\Omega}}
\newcommand{\inc}[2]{\ensuremath{\drept{i}_{#2}^{#1}}}
\newcommand{\mpt}[1]{\ensuremath{\mathscr{P}(#1)}}
\newcommand{\im}[1]{\ensuremath{\mathrm{Im} \hspace*{1pt} #1}}
\newcommand{\centr}[2]{\ensuremath{\drept{C}_{#2}\left(#1\right)}}

\newcommand{\nil}{\ensuremath{\mathbb{0}}}
\newcommand{\spl}[1]{\ensuremath{\overline{#1}}}
\newcommand{\nat}{\ensuremath{\mathbbm{N}}}

\newcommand{\sbg}{\ensuremath{\leqslant}}
\newcommand{\sbgop}{\ensuremath{\leqslant_{\Om}}}
\newcommand{\sbgnop}{\ensuremath{\trianglelefteq_{\Om}}}
\newcommand{\sbgn}{\ensuremath{\trianglelefteq}}
\newcommand{\gensbgop}[1]{\ensuremath{\left\langle #1 \right\rangle_{\Om}}}
\newcommand{\gensbgnop}[1]{\ensuremath{\left\langle \left\langle #1 \right\rangle \right\rangle_{\Om}}}
\newcommand{\inv}[1]{\ensuremath{#1^{-1}}}
\newcommand{\omact}[1]{\ensuremath{\operatorname{\Om-\mathit{#1}}}}
\newcommand{\Sbgnop}[1]{\ensuremath{\rond{S}_{\operatorname{\Om-\drept{Gr}}}^{\drept{n}}\left(#1\right)}}
\newcommand{\Sbgop}[1]{\ensuremath{\mathscr{S}_{\operatorname{\Om-\mathrm{Gr}}}\left(#1\right)}}
\newcommand{\eqv}{\ensuremath{\Leftrightarrow}}
\newcommand{\sdri}{\ensuremath{\mathbf{SDR}_{\drept{i}}}}
\newcommand{\sdrs}{\ensuremath{\mathbf{SDR}_{\drept{s}}}}
\newcommand{\sdrb}{\ensuremath{\mathbf{SDR}_{\drept{b}}}}
\newcommand{\smd}{\ensuremath{\leqslant_{\mathrm{d}}}}
\newcommand{\krn}[1]{\ensuremath{\drept{Ker} \hspace*{1pt} #1}}

\newcommand{\sz}[1]{\ensuremath{\rond{S}_{\drept{s}} \left( #1 \right)}}
\newcommand{\impl}{\ensuremath{\Rightarrow}}
\newcommand{\soc}[1]{\ensuremath{\drept{Soc}\left(#1 \right)}}
\newcommand{\cls}[1]{\ensuremath{\drept{cl} \left( #1 \right)}}
\newcommand{\Sup}[1]{\ensuremath{\rond{Supp} \left( #1 \right)}}
\newcommand{\hmn}[2]{\ensuremath{\mathrm{Hom}^{\mathrm{n}}_{\operatorname{\Omega-{\gr}}} \left(#1, #2 \right)}}
\newcommand{\brac}[1]{\ensuremath{\left(#1\right)}}
\newcommand{\bracc}[1]{\ensuremath{\left\{#1\right\}}}
\newcommand{\card}[1]{\ensuremath{\left|#1 \right|}}

\DeclareMathOperator*{\urds}{\underline{\odot}}

\DeclareMathOperator*{\rprod}{\underline{\times}}
\DeclareMathOperator*{\rds}{\odot}
\DeclareMathOperator*{\mult}{\bullet}

\theoremstyle{definition}
\newtheorem*{theorem*}{Theorem}
\newtheorem{proposition}{Proposition}
\newtheorem*{lemma*}{Lemma}

\title{\Huge{\textbf{The generalisation of some invariants of modules to groups with operators}}}
\author{by\\ Sebastian Cristian Lesnic}
\date{}

\begin{document}

\maketitle

\begin{abstract}
In this article we present a straightforward generalisation to groups with operators of a number of invariants well-known in the theory of modules, having a special bearing on phenomena of semisimplicity. We examine the behaviour of the generalised invariants in relation to the fundamental construction of restricted direct sums and in particular in the context of semisimple groups with operators. We also consider a particular type of morphisms between groups with operators, which naturally preserves the generalised invariants in question and thus shows itself to be an adequate notion for their study in the more general frame considered.      
\end{abstract}

\section{Introduction}

The results presented in this article are excerpted from the PhD thesis of the author, elaborated with the goal of developing a general theory of semisimplicity patterned on the classical one encountered in the category of modules yet applied to the larger category of (not necessarily abelian) groups with operators. Our main references for groups with operators are~\cite{alg1}, ch. I, pg. 29 and~\cite{grth}, pg. 28. We adhere to the set theory formalism expounded on by Bourbaki in~\cite{ens}. We systematise below our preliminary definitions together with our conventions of notation:
\begin{enumerate}

\item For arbitrary set $A$ we write $\Delta_A \eqdef \bracc{\brac{x, x}}_{x \in A}$ for its diagonal and $\mpt{A}$ for its powerset. If $B \subseteq A$ we denote the corresponding inclusion map by $\inc{B}{A}$. Given map $f \colon A \to B$ together with subsets $M \subseteq A$, $N \subseteq B$ such that $f(M) \subseteq N$ we write ${}_{N|}f_{|M}$ for the corresponding restriction of $f$, that is the \textit{unique} map $g \colon M \to N$ such that $\inc{N}{B} \circ g=f \circ \inc{M}{A}$. 

We shall denote the category of sets by $\ens$. Given arbitrary category $\rond{C}$ and two of its objects $X$ and $Y$, we write $\hm{\rond{C}}{X}{Y}$ respectively $\iso{\rond{C}}{X}{Y}$ for the set of all morphisms respectively of all isomorphisms in category $\rond{C}$ from $X$ to $Y$. We express the fact that $X$ and $Y$ are isomorphic in $\rond{C}$ by the formal predicate $X \approx Y \hspace*{3pt} \brac{\rond{C}}$, which is equivalent to the claim $\iso{\rond{C}}{X}{Y} \neq \varnothing$.

We consider a family $B$ of sets indexed by $I$, a family $f \in \displaystyle\prod_{i \in I}\hm{\ens}{A}{B_i}$ of maps and we write $p_i$ for the canonical projection of index $i \in I$ defined on the cartesian product of family $B$. There exists a unique map $g \in \hm{\ens}{A}{\displaystyle\prod_{i \in I}B_i}$ such that $p_i \circ g=f_i$ for every $i \in I$, map which we denote by $\displaystyle\rprod_{i \in I}f_i \eqdef g$ and refer to as the direct product of family $f$ in \textit{restricted} sense.

\item Given an arbitrary set $\Om$ we write $\omact{\gr}$ for the category of $\Om$-groups. In an $\Om$-group $G$, we write ($H \sbgnop G$) $H \sbgop G$ to express the fact that $H$ is a (normal) $\Om$-subgroup of $G$. The corresponding collections of all (normal) $\Om$-groups of $G$ will be denoted by ($\Sbgnop{G}$) $\Sbgop{G}$. For arbitrary subset $X \subseteq G$, the $\Om$-subgroup respectively normal $\Om$-subgroup generated by $X$ in $G$ will be denoted by $\gensbgop{X}$ respectively $\gensbgnop{X}$, and the centraliser of $X$ in $G$ by $\centr{X}{G}$. Given subgroups $K, H \leqslant G$, we use the traditional notation $[H, K]$ for their commutator subgroup.

\item If $H$ is a family of $\Om$-groups indexed by $I$, its \text{restricted direct sum} (cf.~\cite{alg1}, ch. I, pg. 45) will be denoted by $\displaystyle\rds_{i \in I}H_i$. In particular we write $F^{(I)} \eqdef \displaystyle\rds_{i \in I}F$ for arbitrary $\Om$-group $F$. If furthermore the family $H \in \Sbgop{G}^I$ consists of $\Om$-subgroups of $G$ such that $H_j \subseteq \centr{H_i}{G}$ for any pair of \textit{distinct} indices $i, j \in I$ we say that $H$ satisfies (\textbf{CC}) (short for \textit{commutativity conditions}). Given a family of morphisms $f \in \displaystyle\prod_{i \in I}\hm{\omact{\gr}}{F_i}{G}$ such that the associated family of images $\brac{\im{f_i}}_{i \in I} \in \Sbgop{G}^I$ satisfies (\textbf{CC}), there exists a unique morphism $g \in \hm{\omact{\gr}}{\displaystyle\rds_{i \in I}F_i}{G}$ such that $g \circ \iota_i=f_i$ for each $i \in I$, where $\iota_i$ denotes the canonical injection of summand $F_i$ into the restricted direct sum of family $F$. We denote this morphism by $\displaystyle\urds_{i \in I}f_i \eqdef g$ and we introduce the special notation $\theta^H_G \eqdef \displaystyle\urds_{i \in I}\inc{H_i}{G}$. Given a family of morphisms $g \in \displaystyle\prod_{i \in I}\hm{\omact{\gr}}{F_i}{F'_i}$, there exists a unique morphism $h \in \hm{\omact{\gr}}{\displaystyle\rds_{i \in I}F_i}{\displaystyle\rds_{i \in I}F'_i}$ such that $h \circ \iota_i=\iota'_i \circ g_i$, where $\iota_i$ respectively $\iota'_i$ denote the canonical injections of index $i \in I$ into the restricted direct sums of families $F$ respectively $F'$. We denote this morphism by $\displaystyle\rds_{i \in I}g_i \eqdef h$. It is immediate that $\displaystyle\rds_{i \in I}\inc{H_i}{G_i}=\inc{\rds_{i \in I}H_i}{\rds_{i \in I}G_i}$, for any family $G$ of $\Om$-groups and any family $H \in \displaystyle\prod_{i \in I}\Sbgop{G_i}$ of $\Om$-subgroups.

\item In the case of a family $H \in \Sbgop{G}^I$ of $\Om$-subgroups satisfying (\textbf{CC}), we employ the related notations $\sdri\brac{G, H}$, $\sdrs\brac{G, H}$ respectively $\sdrb(G, H)$ as abbreviations of the formal predicates expressing the injectivity, surjectivity respectively bijectivity of $\theta^H_G$ (the abbreviation SDR is the acronym of the French syntagm “somme directe restreinte”, whereas the subscripts are the initials of the adjectives \textbf{i}njective, \textbf{s}urjective respectively \textbf{b}ijective). Relation $\sdri\brac{G, H}$ is equivalent to the mutual independence condition:
\begin{align*}
\brac{\forall i}\brac{i \in I \impl H_i \cap \gensbgop{\bigcup_{j \in I \setminus \bracc{i}}H_j}=\bracc{1_G}}, \label{mi} \tag{MI}
\end{align*}
whereas relation $\sdrs\brac{G, H}$ is equivalent to $G=\gensbgop{\displaystyle\bigcup_{i \in I}H_i}$. By definition it is obvious that $\sdrb\brac{G, H} \eqv \sdri\brac{G, H} \wedge \sdrs\brac{G, H}$. Let us note that given an additional family $K \in \Sbgop{G}^I$ such that $K_i \subseteq H_i$ for every $i \in I$, $K$ also automatically satisfies (\textbf{CC}). If the relations $\sdrs\brac{G, K}$ and $\sdri\brac{G, H}$ are additionally satisfied, it can be shown via condition (MI) that $K=H$ (property (SIE)).

\item Given $\Om$-subgroup $H \sbgop G$, we say $K \sbgop G$ is a \textbf{supplementary} of $H$ in $G$ if the family $F \eqdef \bracc{\brac{1, H}, \brac{2, K}}$ satisfies (\textbf{CC}) as well as the relation $\sdrb\brac{G, F}$. This is equivalent to the conjunction of relations $H, K \sbgnop G$, $H \cap K=\bracc{1_G}$, $HK=G$. We say the $\Om$-subgroup $F \sbgop G$ is a \textbf{direct summand} if it admits a supplementary, and we express this symbolically by $F \smd G$. It follows easily that if $K \sbgnop F \smd G$ then $K \sbgnop G$ (property (NS)). 

\item For arbitrary $\Om$-group $G$ let $\sz{G}$ denote the collection of all \textit{simple} normal $\Om$-subgroups of $G$. In accordance with~\cite{grth}, propositions (3.3.11-14) we recall the equivalence of the following assertions regarding given $\Om$-group $G$:
\begin{enumerate}[label=(\roman*)]
    \item $G$ is isomorphic to the restricted direct sum of a family of \textit{simple} $\Om$-groups
    
    \item $G=\gensbgop{\displaystyle\bigcup\rond{M}}$ for a certain subset $\rond{M} \subseteq \sz{G}$
    
    \item every normal $\Om$-subgroup $F \sbgnop G$ is a direct summand $F \smd G$
\end{enumerate}
and we say that $G$ is \textbf{semisimple} when it satisfies (any of) the above assertions. Given a fixed \textit{simple} $\Om$-group we say $G$ is \textbf{isotypical of type} $S$ if $G \approx {S}^{\brac{I}} \brac{\omact{\gr}}$ for a certain set $I$. 
\end{enumerate}
We briefly mention a slight generalisation of proposition (3.3.12), proved with a similar argument. Given arbitrary sets $A$ and $B$, we write $A \sqcup B \eqdef \brac{\bracc{1} \times A} \cup \brac{\bracc{2} \times B}$ for their disjoint union. We introduce the abbreviation $\spl{M} \eqdef \bracc{\varnothing} \sqcup M$ for arbitrary set $M$. If $x$ is a family of objects indexed by $I$ and $J \subseteq I$ a subset, we write $x_{|J} \eqdef \brac{x_i}_{i \in J}$ for the \textit{subfamily} of $x$ obtained by restricting indices to $J$.
\begin{proposition}
Let $G$ be an arbitrary $\Om$-group, $F \sbgnop G$ a normal $\Om$-subgroup and $H \in \sz{G}^I$ a family of simple normal $\Om$-subgroups indexed by $I$ and such that $G=\gensbgop{F \cup \displaystyle\bigcup_{i \in I}H_i}$. We introduce the “extended” family $H' \in \Sbgnop{G}^{\spl{I}}$ given by $H'_{1\varnothing} \eqdef F$ and $H'_{2i} \eqdef H_i$ for each $i \in I$. There exists a subset $J \subseteq I$ such that the subfamily $H'_{|\spl{J}}$ satisfies (\textbf{CC}) as well as relation $\sdrb\brac{G, H'_{|\spl{J}}}$.
\end{proposition}
As an immediate consequence of this proposition we infer that any normal $\Om$-subgroup and implicitly any quotient of a semisimple (isotypical) $\Om$-group is itself semisimple (isotypical of the same type).  

\section{The generalised invariants and their properties}

\subsection{Definitions}

We give here brief definitions of some of the invariants of groups with operators relevant in the context of restricted direct sums and semisimplicity. While they do constitute straightforward generalisations of the classical invariants introduced for modules, the lack of commutativity of the algebraic structures at hand requires special consideration. Given arbitrary $\Om$-group $G$ we define its \textbf{socle} as:
\[\soc{G} \eqdef \gensbgop{\bigcup \sz{G}}.\]
Since all the simple normal $\Om$-subgroups of $G$ are by construction simple normal $\Om$-subgroups of the socle, it follows that $\soc{G}$ is semisimple since it satisfies assertion (ii) above. By the same token, $G$ is semisimple if and only if $G=\soc{G}$. For fixed simple $\Om$-group $S$ we define \textbf{the isotypical $S$-component} of $G$ as:
\[G_S \eqdef \gensbgop{\bigcup_{\substack{H \in \sz{G}\\ H \approx S \hspace*{1pt} \brac{\omact{\gr}}}}H}.\]
An application of proposition (1) shows in particular that $G_S$ is indeed isotypical of type $S$, justifying the terminology. It is clear that in general $G_S \sbgop \soc{G}$ for any $\Om$-group $G$ and any simple $\Om$-group $S$. Following a syntactic procedure described in~\cite{ens}, ch. II, pg. 47 we define the isomorphism class of $G$, symbolised as $\cls{G}$. This syntactic device has the property that $\cls{\cls{F}}=\cls{F}$ for any $\Om$-group $F$. We proceed to introduce \textbf{the support} of $G$:
\[\Sup{G} \eqdef \bracc{\cls{H}}_{H \in \sz{G}}.\]
If $F \smd G$ is a direct summand we automatically have $\Sup{F} \subseteq \Sup{G}$, cf. property (NS) (\S{1}, paragraph 5). If $G$ is isotypical of type $S$ we easily gather from proposition (1) that $\Sup{G} \subseteq \bracc{\cls{S}}$, with equality occurring if and only if $G \neq \bracc{1_G}$ is nontrivial. It is also clear that $G_S \neq \bracc{1_G} \eqv \cls{S} \in \Sup{G}$ in general. We also note that $G_S=G_{\cls{S}}$, for any simple $\Om$-group $S$.

We say morphism $f \in \hm{\omact{\gr}}{G}{G'}$ is \textbf{normal} if for every normal $\Om$-subgroup $H \sbgnop G$ the direct image $f(H) \sbgnop G'$ is normal in $G'$. We write $\hmn{G}{G'}$ for \textbf{the set of all normal morphisms} from $G$ to $G'$. Normality is preserved under composition of morphisms. Surjective morphisms are automatically normal and so are inclusions of direct summands in their ambient $\Om$-groups, cf. (NS). Thus, more generally any morphism whose image is a direct summand of the target will be normal. 

Given a family $G$ of $\Om$-groups indexed by $I$ we write $\iota_i$ for the canonical injection of index $i \in I$ into the restricted direct sum of family $G$ and we introduce the family of $\Om$-subgroups $G' \eqdef \brac{\gensbgop{\displaystyle\bigcup_{i \in J}\im{\iota_i}}}_{J \in \mpt{I}}$. It is immediate that $G'_J$ and $G'_{I \setminus J}$ are supplementary, by which $G'_J \smd \displaystyle\rds_{i \in I}G_i$. In particular, $G'_{\bracc{i}}=\im{\iota_i}$ is a direct summand and therefore $\iota_i$ is a normal morphism for every $i \in I$.\\

\begin{lemma*} The $\Om$-subgroup generated by a union of \textit{normal} $\Om$-subgroups in a fixed ambient $\Om$-group is itself normal.
\end{lemma*}

\begin{proof}Indeed, let $\rond{M} \subseteq \Sbgnop{G}$ be arbitrary and consider $H \eqdef \gensbgop{\displaystyle\bigcup\rond{M}}$. For any set $M$ we write $\rond{F}(M)$ for the collection of all \textit{finite} subsets of $M$. We consider the extension of the internal multiplication on $G$ to the multiplicative monoid $\brac{\mpt{G}, \cdot}$, where $XY \eqdef \bracc{xy}_{\substack{x \in X\\y \in Y}}$ for any two subsets $X, Y \subseteq G$. Given a finite index set $I$ ordered by total order $T$ and family $X \in \mpt{G}^I$, we write $\displaystyle\mult_{\substack{i \in I\\T}}X_i$ for the product of family $X$ in order $T$ with respect to the monoid $\mpt{G}$. In the case of a family $X$ of pairwise permutable elements ($X_iX_j=X_jX_i$ for any $i, j \in I$), the product of $X$ is independent of the total order specified on $I$ and we simplify the notation to $\displaystyle\mult_{i \in I}X_i$. With these conventions, since $\Sbgnop{G}$ is a \textit{commutative} submonoid of $\mpt{G}$, it is easy to obtain the description:
\[H=\bigcup_{\rond{T} \in \rond{F}\brac{\rond{M}}} \mult_{K \in \rond{T}}K.\]
If $\rond{K} \subseteq \rond{H} \in \rond{F}\brac{\rond{M}}$ we clearly have $\displaystyle\mult_{K \in \rond{K}}K \subseteq \displaystyle\mult_{K \in \rond{H}}K$, hence the set $\bracc{\displaystyle\mult_{K \in \rond{T}}K}_{\rond{T} \in \rond{F}\brac{\rond{M}}}$ is upward-directed under inclusion. As $\Sbgnop{G}$ forms an \textit{algebraic} closure system on $G$ (cf.~\cite{alguniv}, pg. 24), it follows that $H \sbgnop G$.
\end{proof}
In particular we infer that $\soc{G}, G_S \sbgnop G$ for any simple $\Om$-group $S$.

\subsection{Main properties of the invariants}
 We begin with a preliminary result which is of interest in and of itself: 
 
\begin{proposition} \label{prop13} \hfill \\
\begin{enumerate}
\item Let $I$ be a set (of indices), $S$ a family of \textit{simple} \Om-groups indexed by $I$ and such that $S_i \napprox S_j \hspace*{3pt} (\omact{\gr})$ for any \textit{distinct} indices $i, j \in I$ and $G$ an arbitrary \Om-group. Then the family $H \eqdef \brac{G_{S_i}}_{i \in I}$ satisfies (\textbf{CC})
and relation $\sdri(G, H)$ holds.

\item In particular, for arbitrary \Om-group $G$ the family of isotypical components $K \eqdef (G_S)_{S \in \Sup{G}}$ satisfies (\textbf{CC}) and relation $\sdrb(\soc{G}, K)$ holds \textbf{(the socle decomposes as the restricted direct sum of all the nontrivial isotypical components)}.

\item Given arbitrary \Om-groups $G, G'$ and normal morphism $f \in \hmn{G}{G'}$, we have that $f\brac{\soc{G}} \sbgnop \soc{G'}$ and $f\brac{G_S} \sbgnop G'_S$ for any simple \Om-group $S$ \textbf{(normal morphisms “preserve” socles and isotypical components)}.

\item Given index set $I$, arbitrary family $G$ of \Om-groups indexed by $I$ and arbitrary simple \Om-group $S$, the following relations hold:
\begin{align*}
\soc{\rds_{i \in I}G_i}&=\rds_{i \in I} \soc{G_i}\\
\left(\rds_{i \in I}G_i\right)_S&=\rds_{i \in I}(G_i)_S\\
\Sup{\rds_{i \in I}G_i}&=\bigcup_{i \in I} \Sup{G_i}
\end{align*}
\textbf{(restricted direct sums and socles respectively isotypical components “commute”)}. \\
\end{enumerate}
\end{proposition}

\newpage

\begin{proof} \hfill \\
\begin{enumerate}
    \item For ease of notation, let us define $E_i \eqdef \gensbgop{\displaystyle\bigcup_{j \in I \setminus \{i\}} H_j}$ for every $i \in I$ and remark at once that $E_i \sbgnop G$ by virtue of the lemma above, 
since all the $H_j, j \in I$ are normal \Om-subgroups. We argue that in order to prove our claim it will suffice to show that $H_i \cap E_i=\{1_G\}$ for any $i \in I$. Indeed, assuming for a moment that we have achieved this, for indices $i, j \in I, i \neq j$ we have by definition that $H_i \subseteq E_j$ and therefore that $[H_i, H_j] \subseteq H_i \cap H_j \subseteq E_j \cap H_j=\{1_G\}$, since both subgroups $H_i$ and $H_j$ are normal. This means that $H_j \subseteq \centr{H_i}{G}$, in other words family $H$ satisfies (\textbf{CC}). Furthermore, by virtue of condition (MI) above we automatically derive relation $\sdri(G, H)$. 

In order to prove the declared relation of interest, let us also introduce objects:
\begin{align*}
T_i& \eqdef \cls{S_i}\\
\rond{S}_i& \eqdef \{K \in \sz{G}|\ \cls{K}=T_i\}\\
\rond{T}_i &\eqdef \bigcup_{j \in I \setminus \{i\}} \rond{S}_j
\end{align*}
for each $i \in I$. We note that by hypothesis, $i, j \in I$ with $i \neq j$ entails $T_i \neq T_j$. The $H_i$ are semisimple (being isotypical) and by definition we have $H_i=\gensbgop{\displaystyle\bigcup \rond{S}_i}$ for each $i \in I$. By virtue of elementary properties of subgroup generation we see that:
\begin{align*}
E_i&=\gensbgop{\bigcup_{j \in I \setminus \bracc{i}}H_j}\\
&=\gensbgop{\bigcup_{j \in I\setminus \bracc{i}} \gensbgop{\bigcup \rond{S}_i}}\\
&=\gensbgop{\bigcup_{j \in I \setminus \bracc{i}} \brac{\bigcup \rond{S}_i}}\\
&=\gensbgop{\bigcup \rond{T}_i}.
\end{align*}
Since it is clear that by construction $\displaystyle\bigcup_{i \in I} \rond{T}_i \subseteq \sz{E_i}$,
we infer the important fact that $E_i$ is also semisimple for each $i \in I$ (it satisfies assertion (ii) characterising semisimplicity). To be more precise, an application of proposition (1) (with the trivial subgroup playing the role of $F$) yields the existence of a subset $\rond{R} \subseteq \rond{T}_i$ such that $E_i \approx \displaystyle\rds_{K \in \rond{R}}K \quad \brac{\omact{\gr}}$.
Similarly, for any $L \in \sz{E_i}$ on the one hand we have by the same token (proposition (1)) that $L \approx \displaystyle\rds_{K \in \rond{P}}K \hspace*{3pt} \brac{\omact{\gr}}$ for a certain subset $\rond{P} \subseteq \rond{R}$ and on the other that $L$ is simple, which entails $\card{\rond{P}}=1$. This means that $\cls{L}=\cls{K}$ for a certain $K \in \rond{R}$ and therefore entails the inclusion $\Sup{E_i} \subseteq \bracc{T_j}_{j \in I \setminus \bracc{i}}$.

Let us now fix an arbitrary index $i \in I$ and consider the intersection $F \eqdef E_i \cap H_i$. Since both $H_i$ and $E_i$ are normal in $G$, we infer that $F \sbgnop G$ and thus $F \sbgnop E_i, H_i$. Since $F$ is normal in the semisimple $\Om$-group $H_i$, it is itself semisimple. It follows that $F$ is generated by $\displaystyle\bigcup\sz{F}$ and it will thus suffice to show that $\Sup{F}=\varnothing$ in order to infer that $\sz{F}=\varnothing$ and hence $F=\bracc{1_G}$ is trivial. From assertion (iii) characterising semisimplicity we gather that $F \smd E_i, H_i$, from which we consequently deduce that $\Sup{F} \subseteq \Sup{E_i}, \Sup{H_i}$. Keeping in mind that $H_i$ is isotypical of type $S_i$ we deduce that:
\begin{align*}
\Sup{F}&\subseteq \Sup{E_i} \cap \Sup{H_i} \subseteq \bracc{T_j}_{j \in I \setminus \bracc{i}} \cap \bracc{T_i}=\varnothing,
\end{align*}
which means that $F$ is indeed trivial. 

\item In the general context of assertion (1) of the current proposition, making the choice of $\Sup{G}$ to play the role of index set $I$ and of the identity family $\Delta_{\Sup{G}}$ to play that of family $S$, it is clear that for $S, T \in \Sup{G}$ such that $S \neq T$ we have $\cls{S} \neq \cls{T}$, since all the members of $\Sup{G}$ are already isomorphism classes and thus $\cls{S}=S$ for any $S \in \Sup{G}$. Thus, the general condition in the hypothesis of assertion (1) above is satisfied and relation $\sdri(G, K)$ follows. Since $K$ is also a family of $\Om$-subgroups of the socle, it is immediate that relation $\sdri\brac{\soc{G}, K}$ also holds. Let us define:
\[\rond{S}_S \eqdef \{H \in \sz{G}|\ \cls{H}=S\}\]
for any $S \in \Sup{G}$. It is clear that $H \in \rond{S}_{\cls{H}}$ for any $H \in \sz{G}$ and thus:
\[\sz{G}=\bigcup_{S \in \Sup{G}} \rond{S}_S.\]
Since by definition we have that $G_S=\gensbgop{\displaystyle\bigcup \rond{S}_S}$, it follows easily that:
\begin{align*}
\soc{G}&=\gensbgop{\bigcup \sz{G}}\\
&=\gensbgop{\bigcup \brac{\bigcup_{S \in \Sup{G}}\rond{S}_S}}\\
&=\gensbgop{\bigcup_{S \in \Sup{G}}\brac{ \bigcup \rond{S}_S}}\\
&=\gensbgop{\bigcup_{S \in \Sup{G}} G_S},
\end{align*}
which means that relation $\sdrs\brac{\soc{G}, K}$ is also valid, hence so is $\sdrb\brac{\soc{G}, K}$. We record this relation in the form $\soc{G} \approx \displaystyle\rds_{S \in \Sup{G}} G_S \hspace*{3pt} (\omact{\gr})$.

\item Let $f \colon G \to G'$ be a normal morphism. For arbitrary $H \in \sz{G}$ we first remark that $f\brac{H} \sbgnop G'$ by normality and that since $H$ is simple $f$ is either injective or trivial. The two possible cases are thus that $f\brac{H} \approx H \hspace*{3pt} \brac{\omact{\gr}}$ is simple -- which leads to $f\brac{H} \in \sz{G'}$ -- or $f\brac{H}=\{1_{G'}\}$ is trivial. We therefore derive:
\begin{align*}
f\brac{\soc{G}}&=f \brac{\gensbgop{\bigcup \sz{G}}}\\
&=\gensbgop{\bigcup_{H \in \sz{G}} f\brac{H}}\\
&\subseteq \gensbgop{\bigcup \sz{G'} \cup \{1_{G'}\}}\\
&=\soc{G'}.
\end{align*}

As for isotypical components, the situation is similar: if $H \in \sz{G}$ such that $H \approx S \hspace*{3pt} (\omact{\gr})$, then either $f\brac{H} \approx H \approx S \hspace*{3pt} (\omact{\gr})$ or $f\brac{H}$ is trivial. In other words, by introducing:
\begin{align*}
\rond{S}_S &\eqdef \bracc{H \in \sz{G}|\ H \approx S \hspace*{3pt} (\omact{\gr})}\\
\rond{S}'_S &\eqdef \bracc{K \in \sz{G'}|\ K \approx S \hspace*{3pt} (\omact{\gr})}
\end{align*}
we have by definition that $G_S=\gensbgop{\displaystyle\bigcup \rond{S}_S}$ and $G'_S=\gensbgop{\displaystyle\bigcup \rond{S}'_S}$ and we have just seen that:
\[(\forall H)\brac{H \in \rond{S}_S \impl f\brac{H} \in \rond{S}'_S \vee f\brac{H}=\{1_{G'}\}},\]
so that an argument analogous to that concerning socles applies.

\item Let us abbreviate $G' \eqdef \displaystyle\rds_{i \in I} G_i$ and let $\iota_i \in \hm{\omact{\gr}}{G_i}{G'}, p_i \in \hm{\omact{\gr}}{G'}{G_i}$ denote the canonical injection respectively projection of index $i \in I$ associated to the restricted direct sum of family $G$. The canonical projections are obtained by restricting the corresponding projections of the direct product of family $G$ (of which the restricted direct sum is a normal $\Om$-subgroup) and they satisfy the relation $p_j \circ \iota_i=\delta_{ji}$, where:
\begin{align*}
\delta_{ji}&=\begin{cases}
\nil_{G_jG_i}, &i \neq j\\
\id_{G_i}, &i=j,
\end{cases}
\end{align*}
$\nil_{HK}$ being our notation for the null morphism from $K$ to $H$ in category $\omact{\gr}$ (which does possess null objects). Since the canonical projections are all surjective they are also normal and therefore by claim (3) above we have $p_i\brac{\soc{G'}} \subseteq \soc{G_i}$, from which we infer that:
\[\soc{G'} \subseteq \bigcap_{i \in i} \inv{p_i}\brac{\soc{G_i}}=\rds_{i \in I} \soc{G_i}.\]
Conversely, we know that the canonical injections are also normal so by once more referring to the preceding assertion (3) we have $\iota_i\brac{\soc{G_i}} \subseteq \soc{G'}$, which leads to:
\[\soc{G'} \supseteq \gensbgop{\bigcup_{i \in I} \iota_i\brac{\soc{G_i}}}=\rds_{i \in I} \soc{G_i}.\]
In view of the compatibility between normal morphisms and isotypical components, an analogous reasoning applies in the case of isotypical components.

As for the relation between the supports, let us remark that $\Sup{G_i} \subseteq \Sup{G'}$ for any $i \in I$, again by virtue of the normality of the canonical injections. Conversely, let $S \in \Sup{G'}$ be arbitrary and $H \in \sz{G'}$ such that $\cls{H}=S$. As $H$ is simple, we have that for every $i \in I$ one of the two relations $p_i\brac{H} \approx H \hspace*{3pt} (\omact{\gr})$ or $p_i\brac{H}=\{1_{G_i}\}$ must hold. Assuming that the latter relation were valid for each $i \in I$, we would derive that:
\[H \sbgop \bigcap_{i \in I} \inv{p_i}\brac{\{1_{G_i}\}}=\rds_{i \in I}\{1_{G_i}\}=\{1_{G'}\},\]
leading to the absurd conclusion that $H$ is both simple and trivial. Therefore, there must exist $i \in I$ such that $p_i\brac{H} \approx H \approx S \hspace*{3pt} (\omact{\gr})$. Since the canonical projections are surjective and therefore normal, we gather that $p_i\brac{H} \in \sz{G_i}$ and thus $S=\cls{p_i\brac{H}} \in \Sup{G_i}$. \qedhere \\
\end{enumerate}
\end{proof}

On the grounds of assertion (3) in the preceding proposition, we adopt the following syntactic convention: given arbitrary $\Om$-groups $G$ and $G'$, arbitrary \textit{simple} $\Om$-group $S$ and \textit{normal} morphism $f \in \hmn{G}{G'}$ we define \textbf{the $S$-component} of $f$ as $f_S \eqdef {}_{G'_S|}f_{|{\vphantom{G'}G}_S}$. If $G''$ is a third $\Om$-group and $g \in \hmn{G'}{G''}$ another normal morphism, it is immediate that $\brac{g \circ f}_S=g_S \circ f_S$. \\

\begin{theorem*} \label{corprop13} \hfill \\
\begin{enumerate}
\item Let $G$ be an \Om-group, $I$ an arbitrary set of indices, $H \in \Sbgop{G}^I$ a family of \Om-subgroups satisfying (\textbf{CC}) together with relation $\sdrb\brac{G, H}$ and $S$ an arbitrary simple \Om-group. Then the following simultaneously hold:
    \begin{enumerate}
    \item $\soc{H_i} \sbgop \soc{G}$ and $\brac{H_i}_S \sbgop G_S$ for every $i \in I$. 
    
    \item the families $\brac{\soc{H_i}}_{i \in I} \in \Sbgop{\soc{G}}^I$ respectively $\brac{\brac{H_i}_S}_{i \in I} \in \Sbgop{G_S}^I$ satisfy (\textbf{CC}).
    
    \item the relations $\sdrb\brac{\soc{G}, \brac{\soc{H_i}}_{i \in I}}$ and $\sdrb\brac{G_S, \brac{(H_i)_S}_{i \in I}}$ are valid.
    
    \item finally, the relation $\Sup{G}=\displaystyle\bigcup_{i \in I}\Sup{H_i}$ also holds.
    \end{enumerate}

\item For any \Om-group $G$ and any simple \Om-group $S$ it holds that $G_S=(\soc{G})_S$ and that $\Sup{G}=\Sup{\soc{G}}$ \textbf{(the socle encodes all the isotypical components and the support)}.

\item Let $G$ and $G'$ be two semisimple \Om-groups and let us introduce $\rond{S}\eqdef \Sup{G} \cap \Sup{G'}$. The following statements hold:
    \begin{enumerate}
    \item for any normal morphism $f \in \hmn{G}{G'}$ and any simple $\Om$-group $S$, the $S$-component $f_S \in \hmn{{\vphantom{\brac{G'}}G}_S}{G'_S}$ is also a normal morphism. 
    
    \item For any simple $\Om$-group $S$ we introduce the restriction map:
    \begin{align*}
    &\Gamma_S \colon \hmn{G}{G'} \to \hmn{{\vphantom{\brac{G'}}G}_S}{G'_S}\\
    &\Gamma_S(f)=f_S
    \end{align*}
    and consider the direct product in restricted sense $\Phi \eqdef \displaystyle\rprod_{S \in \rond{S}} \Gamma_S$. Then the map:
    \[\Phi \in \iso{\ens}{\hmn{G}{G'}}{\displaystyle\prod_{S \in \rond{S}}\hmn{{\vphantom{\brac{G'}}G}_S}{G'_S}}\]
    is a bijection. \\
    \end{enumerate}
\end{enumerate}
\end{theorem*}

\begin{proof} \hfill \\
\begin{enumerate}
\item Let us abbreviate $\varphi \eqdef \theta^H_G$, an isomorphism by hypothesis. Since we have seen that $\im{\iota_i} \smd \displaystyle\rds_{i \in I}H_i$ and by construction $\varphi \circ \iota_i=\inc{H_i}{G}$, we infer that $H_i \smd G$. It follows that $\sz{H_i} \subseteq \sz{G}$ from which we easily obtain $\soc{H_i} \sbgop \soc{G}$ for any $i \in I$. Furthermore, for $i, j \in I, i \neq j$ we have $\soc{H_j} \subseteq H_j \subseteq \centr{H_i}{G} \subseteq \centr{\soc{H_i}}{G}$ and thus:
\begin{align*}
\soc{H_j} &\subseteq \soc{G} \cap \centr{\soc{H_i}}{G}\\
&=\centr{\soc{H_i}}{\soc{G}},
\end{align*}
which means that the family of \Om-subgroups $\brac{\soc{H_i}}_{i \in I} \in \Sbgop{\soc{G}}^I$ satisfies (\textbf{CC}).

We agree to abbreviate $G' \eqdef \soc{G}$ respectively $H' \eqdef (\soc{H_i})_{i \in I}$ and set up the following commutative diagram (family of diagrams):
\[\xymatrix@+3pc{
&H_i \ar[r]_{\iota_i} \ar@/^1cm/[rr]^{\inc{H_i}{G}} &\displaystyle\rds_{i \in I}H_i \ar[r]_{\theta^H_G} &G\\
&\soc{H_i} \ar@/_1cm/[rr]^{\inc{\soc{H_i}}{\soc{G}}} \ar[u]^{\inc{\soc{H_i}}{H_i}} \ar[r]^{\gamma_i} &\displaystyle\rds_{i \in I} \soc{H_i} \ar[r]^{\theta^{H'}_{G'}} \ar[u]^{\rds_{i \in I}\inc{H'_i}{H_i}} &\soc{G} \ar[u]_{\inc{\soc{G}}{G}}
}\]
for each $i \in I$, where $\iota_i, \gamma_i$ denote the canonical injections into the respective restricted direct sums. From this we infer in particular the commutativity of the right inner rectangle, in other words the relation:
\[\inc{\soc{G}}{G} \circ \theta^{H'}_{G'}=\varphi \circ \left(\rds_{i \in I}\inc{\soc{H_i}}{H_i}\right).\]
Since in particular we have that $\displaystyle\rds_{i \in I} \inc{\soc{H_i}}{H_i}=\inc{\rds_{i \in I} \soc{H_i}}{\rds_{i \in I}H_i}$, we can therefore infer that:
\[\theta^{H'}_{G'}={}_{\vphantom{\left(\theta^H_G\right)}\soc{G}|}{\vphantom{\brac{}}\varphi}_{|\rds_{i \in I}\soc{H_i}}.\]
Since in general the restrictions of injections are themselves injections, we gather that $\theta^{H'}_{G'}$ is also an injection and therefore $\sdri(\soc{G}, H')$ holds. Furthermore, $\theta^H_G$ is by hypothesis an isomorphism so it must bijectively map the source socle onto the target socle. We can thus apply claim (4) of proposition (2) to derive the relation:
\[\soc{G}=\varphi \brac{\soc{\rds_{i \in I}H_i}}=\varphi\brac{\rds_{i \in I}\soc{H_i}}=\im{\left(\theta^{H'}_{G'}\right)}.\]
This relation signifies that $\theta^{H'}_{G'}$ is also surjective, allowing us to conclude that relation $\sdrb(\soc{G}, H')$ holds.

Isomorphisms preserve supports, so by the same token of proposition (2), claim (4)
we infer in the particular case of isomorphism $\theta^H_G$ that:
\[\Sup{G}=\Sup{\rds_{i \in I}H_i}=\bigcup_{i \in I}\Sup{H_i}.\]

As for isotypical components, let us introduce $\rond{S} \eqdef \{K \in \sz{G}|\ K \approx S \hspace*{3pt} (\omact{\gr})\}$ and $\rond{R}_i \eqdef \{K \in \sz{H_i}|\ K \approx S \hspace*{3pt} (\omact{\gr})\}$ for each $i \in I$. Since $H_i \smd G$ we easily gather that $\rond{R}_i \subseteq \rond{S}$ and therefore that:
\[(H_i)_S=\gensbgop{\bigcup \rond{R}_i} \sbgop \gensbgop{\bigcup \rond{S}}=G_S.\]
The rest of the argument is completely analogous to the one presented for socles and relies on the analogous relation of proposition (2), claim (4) involving isotypical components. We omit the details.

\item We have seen in claim (2) of proposition (2) that relation $\sdrb(\soc{G}, (G_S)_{S \in \Sup{G}})$ holds. Therefore, by virtue of claim (1) of the current theorem we infer that:
\[\Sup{\soc{G}}=\bigcup_{S \in \Sup{G}} \Sup{G_S}=\bigcup_{S \in \Sup{G}} \{S\}=\Sup{G},\]
where we have taken into account the fact that by definition for any $S \in \Sup{G}$ the $S$-component $G_S$ is nontrivial and thus $\Sup{G_S}=\bracc{\cls{S}}=\bracc{S}$.

For simplicity we abbreviate $G' \eqdef \soc{G}$ and we introduce: 
\begin{align*}
\rond{S}_S &\eqdef \{H \in \sz{G}|\ H \approx S \hspace*{3pt} (\omact{\gr})\}\\
\rond{S}'_S &\eqdef \{H \in \sz{G'}|\ H \approx S \hspace*{3pt} (\omact{\gr})\}
\end{align*}
for every $S \in \Sup{G}$. It is clear by definition of socles that $\rond{S}_S \subseteq \rond{S}'_S$ and therefore that:
\[G_S=\gensbgop{\bigcup \rond{S}_S} \subseteq \gensbgop{\bigcup \rond{S}'_S}=G'_S.\]
By virtue of proposition (2), claim (2) we know that relation $\sdrs(G', (G_S)_{S \in \Sup{G}})$ holds on the one hand. Applying claim (1) of the same proposition to the $\Om$-group $G'$ and family $\brac{G'_S}_{S \in \Sup{G}}$, we derive that relation $\sdri(G', (G'_S)_{S \in \Sup{G}})$ also holds, on the other hand. On grounds of property (SIE) (\S{1}, paragraph (4)), it follows from these relations that ${\vphantom{\brac{G}}G}_S=G'_S$ for each $S \in \Sup{G}$. For a general simple $\Om$-group $T$, if $\cls{T} \in \Sup{G}$ the conclusion we just reached means that $G_T=G'_T$. If however $\cls{T} \notin \Sup{G}=\Sup{\soc{G}}$ it follows trivially that $G_T=G'_T=\bracc{1_G}$.

    \item 
        \begin{enumerate}    
        \item    Let us set $\rond{P}\eqdef \Sup{G}$ and $\rond{Q} \eqdef \Sup{G'}$. In claim (3) of proposition (2) we have seen that any normal morphism $f \in \hmn{G}{G'}$ restricts between $G_S$ and $G'_S$ for any simple \Om-group $S$. Since the $S$-component of $G$ is normal in a semisimple $\Om$-group, it is a direct summand. Thus, by virtue of property (NS), if $H \sbgnop G_S$ we gather $H \sbgnop G$ and by normality of $f$ we have $f_S(H)=f(H) \sbgnop G'$ and at the same time $f(H) \sbgop G'_S$. This entails $f_S(H) \sbgnop G'_S$ and means the $S$-component $f_S$ is itself normal.

        \item Let us argue that $\Phi$ is bijective. Indeed, assuming that $\Phi(f)=\Phi(g)$ for two normal morphisms $f, g \in \hmn{G}{G'}$, we gather that:
        \[\bigcup_{S \in \rond{S}}G_S \subseteq \drept{Eq}(f, g)\eqdef \inv{\brac{f \rprod g}}\brac{\Delta_{G'}}.\]
        Since by virtue of proposition (2), claim (2) we know that $G$ is generated by all of its isotypical components, in order to show that $f=g$ it will suffice to prove that for any $S \in \rond{P} \setminus \rond{S}$ we have $G_S \sbgop \drept{Eq}(f, g)$. However, for such $S \in \rond{P} \setminus \rond{S}=\rond{P} \setminus \rond{Q}$ we have $f\brac{G_S} \sbgop G'_S$ and since $S \notin \rond{Q}$ it follows that $G'_S=\{1_{G'}\}$, by which $G_S \sbgop \krn{f}$. This equally applies to morphism $g$ and means that $G_S \sbgop \krn{\brac{f \rprod g}} \sbgop \drept{Eq}(f, g)$. This argumentation establishes the injectivity of $\Phi$. 
    
        As to surjectivity, let $g \in \displaystyle\prod_{S \in \rond{S}} \hmn{{\vphantom{\brac{G'}}G}_S}{G'_S}$ be arbitrary and let us introduce the family of morphisms:
        \[f' \eqdef \brac{\inc{G'_S}{G'} \circ g_S}_{S \in \rond{S}} \cup \brac{\nil_{G' G_S}}_{S \in \rond{P} \setminus \rond{S}}.\]
        Since the family of isotypical components $(G'_S)_{\vphantom{G'_S}S \in \rond{Q}}$ satisfies (\textbf{CC}) cf. proposition (2), claim (1), it is clear by definition that the family of images $\brac{\im{f'_S}}_{S \in \rond{P}}$ satisfies (\textbf{CC}) as well, and we can thus consider morphism:
        \[h\eqdef \displaystyle\urds_{S \in \rond{P}}f'_S \in \hm{\omact{\gr}}{\displaystyle\rds_{S \in \rond{P}}G_S}{G'}.\]
        By virtue of claim (2) in the preceding proposition we have that $\varphi \eqdef \theta^H_G$ is an isomorphism, where we have abbreviated $H\eqdef \brac{G_S}_{S \in \rond{P}}$. For arbitrary $S \in \rond{P}$ let $\iota_S$ denote the canonical injection of index $S \in \rond{P}$ into the restricted direct sum of family $H$. Let us finally consider morphism $f \eqdef h \circ \inv{\varphi} \in \hm{\omact{\gr}}{G}{G'}$ and let us argue that it is normal and that $\Phi(f)=g$. 
    
        Let $F \sbgnop G$ be an arbitrary normal \Om-subgroup. We gather that $F$ is also semisimple and that $\Sup{F} \subseteq \Sup{G}$, since it follows that $F \smd G$ (assertion (iii) characterising semisimplicity). We thus have:
        \[\varphi\brac{\rds_{S \in \rond{P}}F_S}=\varphi\brac{\gensbgop{\bigcup_{S \in \rond{P}} \iota_S\brac{F_S}}}=\gensbgop{\displaystyle\bigcup_{S \in \rond{P}}F_S}=\gensbgop{\bigcup_{S \in \Sup{F}}F_S}=F\]
        or equivalently $\inv{\varphi}\brac{F}=\displaystyle\rds_{S \in \rond{P}}F_S$. Applying $h$ to this relation we have by definition:
        \[f\brac{F}=h\brac{\rds_{S \in \rond{P}}F_S}=\gensbgop{\bigcup_{S \in \rond{P}}f'_S\brac{F_S}}=\gensbgop{\bigcup_{S \in \rond{S}}g_S\brac{F_S}}.\]
        We recall that the morphisms $g_S$ are normal and infer that $g_S\brac{F_S} \sbgnop G'_S$ is normal in $G'_S$ for any $S \in \rond{S}$. Since $G'_S \smd G'$ (isotypical components are always normal, hence direct summands in the semisimple $\Om$-group $G'$), we infer by virtue of property (NS) that $g_S\brac{F_S} \sbgnop G'$ is thus normal in $G'$. By virtue of the lemma which concludes subsection (2.1) we finally gather that $f\brac{F} \sbgnop G'$, establishing the normality of $f$.
    
        Let us introduce the corestriction:
        \[\gamma_S\eqdef \inv{\brac{{}_{\im{\iota_S}|}\iota_S}} \in \iso{\omact{\gr}}{\im{\iota_S}}{G_S}.\]
        Bearing in mind that $\inv{\varphi}\brac{G_S}=\im{\iota_S}$ we have by definition the relations:
        \begin{align*}
        \id_{G_S}&={}_{G_S|}\brac{\inc{G_S}{G}}\\
        &={}_{G_S|}\brac{\varphi \circ \iota_S}\\
        &={}_{G_S|}\varphi_{|\im{\iota_S}} \circ {}_{\im{\iota_S}|}\iota_S,
        \end{align*}
        which lead to $\gamma_S={}_{G_S|}\varphi_{|\im{\iota_S}}$ and $f'_S=h \circ \iota_S=h_{|\im{\iota_S}} \circ {}_{\im{\iota_S}|}\iota_S$, further entailing $f'_S \circ \gamma_S=h_{|\im{\iota_S}}$. Since it is also clear that $\inv{\brac{{}_{G_S|}\varphi_{|\im{\iota_S}}}}={\vphantom{\inv{\varphi}}}_{\im{\iota_S}|}\brac{\inv{\varphi}}_{|G_S}$, it follows that $\gamma_S \circ {\vphantom{\inv{\varphi}}}_{\im{\iota_S}|}\brac{\inv{\varphi}}_{|G_S}=\id_{G_S}$ and therefore -- under the added assumption $S \in \rond{S}$ -- that:
        \begin{align*}
        f_{|G_S}&=\brac{h \circ \inv{\varphi}}_{|G_S}\\
        &={\vphantom{\brac{\inv{\varphi}}}h}_{|\im{\iota_S}} \circ {\vphantom{\brac{\inv{\varphi}}}}_{\im{\iota_S}|}\brac{\inv{\varphi}}_{|G_S}\\
        &=f'_S \circ \gamma_S \circ {\vphantom{\brac{\inv{\varphi}}}}_{\im{\iota_S}|}\brac{\inv{\varphi}}_{|G_S}\\
        &=f'_S \circ \id_{G_S}\\
        &=\inc{G'_S}{G'} \circ g_S,
        \end{align*}
        which means precisely that $g_S={}_{G'_S|}f_{|{\vphantom{G'_S}G}_S}$. This shows that $\Phi(f)=g$ and thus $\Phi$ is also surjective. \qedhere \\
        \end{enumerate}
        
\end{enumerate}
\end{proof}

Let us point out a phenomenon that can only be encountered in the case of noncommutative groups with operators. The content of claim (2) in the theorem is that arbitrary $\Om$-group $G$ only has the same support as its socle, not necessarily that they have the same collection of simple normal $\Om$-subgroups. More explicitly, (counter)examples can be given of instances where $H \sbgnop \soc{G}$ is simple without $H \sbgop G$ being normal in $G$ however. 

We indicate here one pattern for such an instance: consider $\Om=\varnothing$ (so that we are working in the ordinary category \gr) and $G$ a group such that $A \eqdef \soc{G}$ is \textit{abelian and simple} however non-central (in other words, $A \setminus \drept{Z}(G) \neq \varnothing$). Then by virtue of proposition (2), claim (4) we have $\soc{G \times G}=A \times A$, and it is easy to see that $\Delta_A \in \sz{\soc{G \times G}}$. However, in view of the equivalence relation:
\[\Delta_A \sbgn G \times G \eqv A \sbg \drept{Z}(G)\]
we deduce from the given hypothesis that $\Delta_A$ is not normal in the ambient group $G \times G$. 

We agree to write $\Sigma(A)$ for the symmetric group of arbitrary set $A$, $\drept{Alt}(A)$ for the corresponding alternating group and $\Sigma_n \eqdef \Sigma([1, n])$, $\drept{A}_n \eqdef \drept{Alt}([1, n])$ for the symmetric respectively alternating groups of degree $n$, in the particular case $A=[1, n] \eqdef \{k \in \nat| 1 \leqslant k \leqslant n\}$, where $n \in \nat$ is arbitrary. With this terminology, a concrete illustration of the above pattern is to be found in the case $G=\Sigma_3$, $A=\drept{A}_3$. \\


\end{document}